\documentclass[11pt]{amsart}

\usepackage{graphicx}
\usepackage[USenglish,english]{babel}
\usepackage[T1]{fontenc}
\usepackage[latin1]{inputenc}
\usepackage{amsfonts}
\usepackage{amssymb}
\usepackage{amsthm}
\usepackage{amsmath}
\usepackage{tikz}
\usepackage{float}
\usepackage{enumerate}
\usepackage{accents}
\usepackage{mathtools}
\usepackage{mathrsfs}
\usepackage{comment}
\usepackage{afterpage}
\usepackage{bbm}
\usepackage{dsfont}
\usepackage{stmaryrd}
\usepackage{hyperref}
\usepackage{mleftright}
\usepackage{subfig}

\theoremstyle{plain}
\newtheorem{thm}{Theorem}[section]
\newtheorem{lem}[thm]{Lemma}
\newtheorem{prop}[thm]{Proposition}
\newtheorem{cor}[thm]{Corollary}
\newtheorem*{thm*}{Main Theorem}
\newtheorem*{prop*}{Proposition}
\newtheorem*{cor*}{Corollary}

\theoremstyle{definition}

\newtheorem{exam}[thm]{Example}
\newtheorem{rem}[thm]{Remark}
\newtheorem*{quest*}{Question}
\newtheorem{claim}{Claim}

\DeclareMathOperator{\diam}{diam}

\newcommand{\R}{\mathbb{R}}

\newcommand{\cu}{\subseteq}

\newcommand{\lk}{\mathrm{lk}}

\newcommand{\CAT}{{\rm CAT(0)}}

\newcommand{\Hyp}{\mathcal{HYP}}
\newcommand{\FG}{\mathcal{FG}}
\newcommand{\FQ}{\mathcal{FQ}}
\newcommand{\supp}{\mathrm{supp}}

\title{Comparing topologies on the Morse boundary and quasi-isometry invariance}
\author{Merlin Incerti-Medici}

\begin{document}

\maketitle

\begin{abstract}
We compare several topologies on the Morse boundary $\partial_M Y$ of a $\CAT$ cube complex $Y$. In particular, we show that the two topologies introduced by Cashen and Mackay are not equal in general and provide a new description of one of them in the language of cube complexes. As a corollary, we obtain a new approach to tackle the question whether the visual topology induces a quasi-isometry-invariant topology on the Morse boundary. This leads to an obstruction to quasi-isometry-invariance in terms of the behaviour of geodesics under quasi-isometries.
\end{abstract}

\tableofcontents

%---------------------------------------------------------------------------------------------------------------------------------------

%INTRODUCTION

%---------------------------------------------------------------------------------------------------------------------------------------

\section{Introduction} \label{sec:Introduction}

Boundaries at infinity are a common tool in the study of large scale geometric properties. When a group acts geometrically on a metric space, one can study the relations between the group and the boundary of the metric space. A particularly fruitful instance of this is the case of a Gromov-hyperbolic metric space and its visual boundary. Since any quasi-isometry between hyperbolic metric spaces induces a homeomorphism on the visual boundary, we can define the visual boundary of a hyperbolic group as a topological space up to homeomorphism (see \cite{Gromov-HypGps}). This is no longer true, when the group is acting on a non-positively curved space, e.\,g.\,a $\CAT$ space. In \cite{CrokeKleiner}, Croke and Kleiner provided an example of a group acting geometrically on two different quasi-isometric $\CAT$ spaces which have non-homeomorphic visual boundaries. So we cannot associate the visual boundary as a topological space up to homeomorphism to a group.

Charney and Sultan introduced an alternative boundary for $\CAT$ spaces, the contracting boundary, whose homeomorphism class is invariant under quasi-isometry (see \cite{CharneySultan}). Cordes generalised their concept to proper geodesic metric spaces, introducing the Morse boundary (see \cite{Cordes}). In either case, this boundary consists of all points in the visual boundary, which are represented by geodesic rays that have similar properties to geodesic rays in hyperbolic spaces, indicating that these are the `hyperbolic' directions in the space under consideration (see section \ref{sec:Preliminaries} for definitions).

Charney and Sultan equipped the contracting boundary with a quasi-isometry invariant topology by using a direct limit construction. A similar construction is done to define a quasi-isometry-invariant topology on the Morse boundary in greater generality (see \cite{CordesHume}). This allows us to define the Morse boundary of a group as the homeomorphism class of such a direct limit. However, this topology is not first countable in general. Concretely, this topology is not first countable for the group $\mathbb{Z}^2 * \mathbb{Z}$, an example also known as the `tree of flats' (see \cite{Murray}).

Cashen and Mackay introduced two new, coarser topologies (\cite{CashenMackay}). Both topologies are based on the notion of fellow-traveling paths, a useful concept in the study of the visual boundary of hyperbolic spaces. Cashen and Mackay generalized this concept to proper, geodesic metric spaces to introduce the topology of fellow-traveling geodesics $\FG$ and the topology of fellow-traveling quasi-geodesics $\FQ$, the second of which is invariant under quasi-isometries and thus allows us to turn the Morse boundary into a topological invariant of a group (see section \ref{subsec:Morse boundary} for definitions of the Morse boundary, $\FG$ and $\FQ$). If a group acts geometrically on a proper geodesic metric space $Y$, then $(\partial_M Y, \FQ)$ is metrizable and second countable. However, Cashen and Mackay's proof of metrizability relies on the Urysohn metrisation theorem, whose proofs do not provide a metric that is easy to compute. Giving an explicit construction of a metrization of $\FQ$ that stays faithful to the geometric context of Morse boundaries is still an open problem at the time of writing.\\

In this article, we restrict our attention to the Morse boundary of $\CAT$ cube complexes. Cube complexes were introduced by Gromov in \cite{Gromov-HypGps} and have become a central object in geometric group theory over the last decade due to their fruitfulness in solving problems in group theory and low-dimensional topology and due to the fact that many interesting groups are cubulable, i.\,e.\,they act properly and cocompactly on a $\CAT$ cube complex. The class of groups that are cubulable includes Right-angled Artin groups, hyperbolic $3$-manifold groups (\cite{BergersonWise}), most non-geometric $3$-manifold groups (\cite{PiotrWise14}, \cite{HagenPiotr}, \cite{PiotrWise18}), small cancelation groups (\cite{Wise}) and many others. Cubulated groups played a key role in Agol's and Wise's proof of the virtual Haken and virtual fibered conjecture (\cite{Wise11}, \cite{Agol}).\\

When considering a $\CAT$ cube complex $Y$, one can define the following topology on the visual boundary: Fix a vertex $o \in Y$ as a base point and let $h_1, \dots, h_n$ be distinct hyperplanes in $Y$. Denote the visual boundary of $Y$ by $\partial_{\infty} Y$. Define the set

\begin{equation*}
\begin{split}
V_{o, h_1, \dots, h_n} := \{ \xi \in \partial_{\infty} Y | \text{The unique} & \text{ geodesic representative of $\xi$ based}\\
& \text{ at $o$ crosses the hyperplanes $h_1, \dots, h_n$.} \}.
\end{split}
\end{equation*}

The collection $\{ V_{o, h_1, \dots, h_n} \}_{n, h_1, \dots, h_n}$ forms the basis of a topology on $\partial_{\infty} Y$. We denote the subspace topology on the Morse boundary induced by this topology by $\Hyp$.

A comparison of $\FG$ and $\Hyp$ yields

\begin{thm} \label{thm:TheoremA}
Let $Y$ be a uniformly locally finite $\CAT$ cube complex. Then the topologies $\FG$ and $\Hyp$ coincide on the Morse boundary.
\end{thm}

It is shown in \cite{CashenMackay} that $\FG \subset \FQ$ and a sufficient condition for equality is given. We show that equality does not hold in general.

\begin{thm} \label{thm:TheoremB}
There exists a uniformly locally finite $\CAT$ cube complex $Y$ that admits a geometric action by a group and $\FG \neq \FQ$ on $\partial_M Y$.
\end{thm}

The counter-example used to prove Theorem \ref{thm:TheoremB} is the same Right-angled Artin group that was already used by Croke and Kleiner to show that the visual boundary is not a quasi-isometry invariant for $\CAT$ spaces (see section \ref{sec:Comparison} for details). Our construction to prove that this is a counter-example shows in particular that no minor adjustments to Cashen and Mackays concept of fellow-traveling will change the inequality of $\FG$ and $\FQ$.\\

Combining Theorem \ref{thm:TheoremA} with results from \cite{CashenMackay} and \cite{BeyFio} (see section \ref{sec:VisualTopology} for details), we obtain 

\begin{cor} \label{cor:Corollary1}
Let $Y$ be a uniformly locally finite $\CAT$ cube complex. Then the following topologies on $\partial_M Y$ coincide:

\begin{enumerate}

\item The subspace topology induced by the visual topology

\item $\FG$

\item $\Hyp$

\item The topology induced by the Roller boundary.

\end{enumerate}

In particular one of them is invariant under quasi-isometries if and only if the others are.

\end{cor}

Note that the visual topology and $\Hyp$ do not coincide on all of $\partial_{\infty} Y$ in general; $\mathbb{R}^2$ with its standard cubulation provides an easy counter example.

In \cite{Cashen}, Cashen showed that the subspace topology induced by the visual topology is not invariant under quasi-isometries in general. Since his counter examples do not admit a cocompact action by isometries, he raised the question whether the Morse boundary with the visual topology is invariant under quasi-isometry when the spaces in question admit a cocompact action by isometries. Corollary \ref{cor:Corollary1} provides us with a new tool to study this question for cubulable groups.\\

One necessary condition for the quasi-isometry invariance of $\FG$ is that it is independent of the choice of base point. It turns out that $\FG$ (and thus the other three topologies in Corollary \ref{cor:Corollary1}) are independent of the choice of base point if the metric space $Y$ is $\CAT$. We provide a counter example for non-$\CAT$ spaces. In fact, our counter example will be a finitely generated small cancellation group.\\

Restricting our attention to a special class of $\CAT$ cube complexes (which includes the universal covering of the Salvetti complex of any RAAG) allows us to give an explicit construction of a metric on $\partial_M Y$ that induces the topology $\FG$. This metric depends on the choice of a base point. We show that the cross ratio induced by this metric also depends on the choice of a base point and we will compare these cross ratios to the one introduced in \cite{BeyFioMedici}.\\

The remainder of the paper is organised as follows. In section \ref{sec:Preliminaries}, we recall basic notions and facts about Morse boundaries, $\CAT$ cube complexes and RAAGs. We prove that for $\CAT$ spaces, $\FG$ is independent of the choice of base point and provide a counter example for non-$\CAT$ spaces. In section \ref{sec:Comparison}, we will prove Theorem \ref{thm:TheoremA} and Theorem \ref{thm:TheoremB}. In section \ref{sec:Metrizing}, we introduce a metric on $\partial_M Y$ that induces the topology $\FG$ for a special class of $\CAT$ cube complexes and compare this metric to notions introduced in \cite{BeyFioMedici}. In section \ref{sec:VisualTopology}, we will discuss Corollary \ref{cor:Corollary1} and finish with a short analysis of the question whether $\Hyp$ is invariant under quasi-isometries.\\

The author is grateful to Ruth Charney, Matthew Cordes, Dominik Gruber, Viktor Schroeder, Alessandro Sisto and Davide Spriano for many discussions and helpful advice. The author thanks Jonas Beyrer and Elia Fioravanti for pointing out a mistake in an earlier version of this paper. The author also thanks Christopher Cashen and John Mackay for their comments on Proposition \ref{prop:basepointindependence} and for providing Example \ref{exam:counterexample} and Matthew Cordes for providing figure \ref{fig:Torus}.

%-------------------------------------------------------------------------------------------------------------------------------------

%PRELIMINARIES

%-------------------------------------------------------------------------------------------------------------------------------------

\section{Preliminaries} \label{sec:Preliminaries}

\subsection{The Morse boundary} \label{subsec:Morse boundary}

For a more thorough introduction to Morse boundaries and their properties, see \cite{Cordes}.

Let $(Y,d)$ be a proper, geodesic metric space. Two quasi-geodesic rays $\gamma$, $\gamma'$ are called asymptotically equivalent, if they have uniformly bounded distance, i.\,e. $d(\gamma(t), \gamma'(t)) \leq B < \infty$ for all $t \geq 0$ and some $B \in \mathbb{R}$. The visual boundary $\partial_{\infty} Y$ of $(Y,d)$ is defined to be the set of all equivalence classes of quasi-geodesic rays. Equivalently, one may fix a base point $o$ and consider equivalence classes of quasi-geodesic rays that start at $o$.

Given a subset $S \subset Y$ and $R \geq 0$, we denote the {\it $R$-neighbourhood} of $S$ by
\[ N_R(S) := \{ y \in Y | \exists s \in S : d(y,s) \leq R \}. \]

Let $N : \mathbb{R}_{\geq 1} \times \mathbb{R}_{\geq 0} \rightarrow \mathbb{R}_{\geq 0}$ be a function. A set $S \subset Y$ is called {\it $N$-Morse} if every $(K,C)$-quasi-geodesic $\gamma$ whose endpoints lie in $S$ is contained in the $N(K,C)$-neighbourhood of $S$, i.\,e. $\gamma \subset N_{N(K,C)}(S)$. We call a set Morse, if there exists a function $N$, such that the set is $N$-Morse. We define the {\it Morse boundary} $\partial_M Y \subset \partial_{\infty} Y$ to be the set of all points in the visual boundary that can be represented by a quasi-geodesic, whose image is Morse. Note that any set, that has bounded distance of an $N$-Morse set is $N'$-Morse for some function $N'$ which depends only on $N$ and the Hausdorff-distance between the two sets. In particular, all representatives of a point in the Morse boundary are Morse. Moreover, one can show that for any point $\xi \in \partial_M Y$ and any base point $o \in Y$, there exists some $N$, such that all geodesic representatives of $\xi$ are $N$-Morse.

There is an equivalent characterisation of the Morse-property. Let $\rho : \mathbb{R}_{\geq 0} \rightarrow \mathbb{R}_{\geq 0}$ be a non-decreasing, eventually non-negative function, which is sublinear, i.\,e. $\lim_{r \rightarrow \infty} \frac{\rho(r)}{r} = 0$. Consider a closed set $S \subset Y$ and $y \in Y$. We denote the set of closest points to $y$ in $S$ by $\pi_S(y)$ and call $\pi_S$ the {\it closest point projection onto $S$}, even though $\pi_S$ is, strictly speaking, not a map from $Y$ to $S$. A closed set $S \subset Y$ is called {\it $\rho$-contracting}, if for all $x, y \in Y \diagdown S$ with $d(x,y) < d(S, y)$, we have that $\diam(\pi_S(x), \pi_S(y)) \leq \rho(d(S,y))$. In other words, for any point $y \in Y \diagdown S$, the projection of the largest open ball $B$, centered at $y$ that does not intersect $S$, onto $S$ has diameter bounded by $\rho(r)$ where $r$ denotes the radius of $B$. We call a closed set contracting, if it is $\rho$-contracting for some non-decreasing, eventually non-negative, sublinear function $\rho$.

In \cite{ACGH}, the authors proved that for every function $N$, there exists a function $\rho$, depending only on $N$, such that any closed $N$-Morse set is $\rho$-contracting. Conversely, for every $\rho$ there exists an $N$, such that every closed $\rho$-contracting is $N$-Morse. Thus, we see that the contracting boundary, which is defined to be the set of all points in the visual boundary, that admit a $\rho$-contracting representative (and hence all representatives are contracting), is the same as the Morse boundary. If $Y$ is $\CAT$, a geodesic is Morse if and only if there exists a constant $D$ such that the geodesic is $D$-contracting (cf. \cite{BestvinaFujiwara, Sultan}). 

Note that so far, we have only introduced the Morse boundary as a set. We will now recall some constructions and results from \cite{CashenMackay} in order to define the topologies $\FG$ and $\FQ$. We begin with their key-observation on the divergence-behaviour of quasi-geodesics from contracting sets. Given a sublinear function $\rho$, $K \geq 1$ and $C \geq 0$, we define
\[ \kappa(\rho, K, C) := \max(3K^2, 3C, 1 + \inf\{ R \, | \, \forall r \geq R, \rho(r) \leq 3K^2 r \} )\]
and
\[ \kappa'(\rho, K, C) := (K^2 + 2) (2\kappa(\rho, K, C) + C). \]

\begin{prop}[Corollary 4.3 in \cite{CashenMackay}] \label{prop:CashenMackay}

Let $Z$ be $\rho$-contracting and let $\beta$ be a continuous $(K',C')$-quasi-geodesic ray with $d(\beta_0, Z) \leq \kappa(\rho, K', C')$. There are two possibilities:

\begin{enumerate}

\item[(1)] The set $\{ t | d(\beta_t, Z) \leq \kappa(\rho, K', C') \}$ is unbounded and $\beta$ is contained in the $\kappa'(\rho, K', C')$-neighbourhood of $Z$.

\item[(2)] There exists a last time $T_0$ such that $d(\beta_{T_0}, Z) = \kappa(\rho, K', C')$ and:

\[ \forall t, d(\beta_t, Z) \geq \frac{1}{2K'}(t - T_0) - 2(C' + \kappa(\rho, K', C')). \]

\end{enumerate}

\end{prop}

This motivates the following definition. Fix $o \in Y$. Let $\xi \in \partial_M Y$. Since $\xi$ is in the contracting boundary, there exists $\rho$ such that all geodesic representatives of $\xi$ based at $o$ are $\rho$-contracting. Choose one geodesic representative $\gamma \in [\xi]$ that is based at $o$. Let $R \geq 0$ We say that a $(K,C)$-quasi-geodesic $\beta$ fellow travels along $\gamma$ for distance $R$, if $\beta \cap N_{\kappa(\rho, K, C)}(\gamma \diagdown B_{R}(o)) \neq \emptyset$, where $B_{R}(o)$ denotes the open ball of radius $R$ centered at $o$. We define the set

\begin{equation*}
\begin{split}
U_{o, R}(\xi) := \{ \eta \in \partial_M Y |  &\text{ All quasi-geodesic representatives of $\eta$ that are}\\
&\text{ based at $o$ fellow-travel along $\gamma$ for distance $R$.} \}.
\end{split}
\end{equation*}

Analogously, we define

\begin{equation*}
\begin{split}
V_{o, R}(\xi) := \{ \eta \in \partial_M Y |  &\text{ All geodesic representatives of $\eta$ that are}\\
&\text{ based at $o$ fellow-travel along $\gamma$ for distance $R$.} \}.
\end{split}
\end{equation*}

Cashen and Mackay show that the two families $\{U_{o, R}(\xi)\}_{R, \xi}$ and $\{V_{o, R}(\xi)\}_{R, \xi}$ are neighbourhood bases of two topologies. We denote the topology induced by $\{U_{o, R}(\xi) \}_{R, \xi}$ by $\FQ$ and the topology induced by $\{V_{o, R}(\xi)\}_{R, \xi}$ by $\FG$. In other words, a set $U$ is open in $\FQ$ if and only if for every point $\xi \in U$ there exists some $R > 0$, such that $U_{o, R}(\xi) \subset U$. The analogous statement defines $\FG$.

\begin{prop} \label{prop:basepointindependence}
$\FQ$ is independent of the choice of the base point $o$. If $Y$ is $\CAT$, then $\FG$ is also independent of the choice of $o$.
\end{prop}

While the first part of this proposition was proven in \cite{CashenMackay}, we need to provide a proof for the second part. Before we do so, we state a key technical lemma from \cite{CashenMackay}.

\begin{lem}[Lemma 4.6 of \cite{CashenMackay}]
Let $\alpha$ be a $\rho$-contracting geodesic ray and let $\beta$ be a continuous $(K,C)$-quasi-geodesic ray with $\alpha(0) = \beta(0) = o$. Given some $R$ and $J$, suppose there exists a point $x \in \alpha$ with $d(x,o) \geq R$ and $d(x, \beta) \leq J$. Let $y$ be the last point on the subsegment of $\alpha$ between $o$ and $x$ such that $d(y, \beta) = \kappa(\rho, K, C)$. There is a constant $M \leq 2$ and a function $\lambda(\phi, p, q)$ defined for sublinear $\phi$, $p \geq 1$ and $q \geq 0$ such that $\lambda$ is monotonically increasing in $p$ and $q$ and

\[ d(x,y) \leq MJ + \lambda(\rho, K, C). \]

Thus,

\[ d(o,y) \geq R - MJ - \lambda(\rho, K, C). \]

\end{lem}

We reformulate this Lemma in a way that will be a bit more handy to use.

\begin{cor} \label{cor:Reformulation}
Let $\alpha$ be a $\rho$-contracting geodesic ray based at $o$, $K \geq 1$, $C \geq 0$, $J \geq 0$, $r \geq 0$. Then, there exists a constant $R(\rho, K, C, r, J)$ such that all $(K,C)$-quasi-geodesic rays based at $o$ satisfying $(*)$ are fellow-traveling along $\alpha$ for at least distance $r$.

\begin{enumerate}
\item[$(*)$] There exists a point $p \in \alpha$ satisfying $d(o,p) \geq R(\rho, K, C, r, J)$ and $d(\beta, p) \leq J$.
\end{enumerate}
\end{cor}

Note that condition $(*)$ can be thought of as a version of fellow-traveling, where the quasi-geodesic $\beta$ is $J$-fellow-traveling along the geodesic $\alpha$ for distance $R = R(\rho, K, C, r, J)$.

\begin{proof}[Proof of Proposition \ref{prop:basepointindependence}]
For $\FQ$, this has been proven in \cite{CashenMackay}. Suppose that $Y$ is $\CAT$. Let $o$, $o' \in Y$, $\xi \in \partial_M Y$. There exists a sublinear function $\rho$, such that all geodesic representatives of $\xi$ based at $o$ or $o'$ are $\rho$-contracting. To distinguish the neighbourhood-bases with respect to $o$ and $o'$, we will denote the neighbourhoods by $U_{o, R}(\xi)$ and $U_{o', R}(\xi)$ respectively. Let $\alpha \in \xi$ be a geodesic based at $o$ and $\alpha' \in \xi$ a geodesic based at $o'$.

Let $R \geq 0$. We need to find $R'$ such that $U_{o', R'}(\xi) \subset U_{o, R}(\xi)$. Put

\[ R' := R\left(\rho, 1, 0, R, [ \kappa'(\rho, 1, d(o,o')) + \kappa(\rho, 1, 0) + d(o,o') ] \right) \; + \, \kappa'(\rho, 1, d(o,o')) + d(o,o')\]
\\
using the function $R(\cdot, \cdot, \cdot, \cdot, \cdot)$ from Corollary \ref{cor:Reformulation}. We claim that $U_{o', R'}(\xi) \subset U_{o, R}(\xi)$.

Let $\eta \in U_{o', R'}(\xi)$, $\beta \in \eta$ a geodesic representative based at $o$ and $\beta' \in \eta$ a geodesic representative based at $o'$. We know that $\beta' \cap N_{\kappa(\rho, 1, 0)}(\alpha' \diagdown B_{R'}(o)) \neq \emptyset$ and thus, there exist points $p' \in \alpha' \diagdown B_{R'}(o)$ and $q' \in \beta'$ such that $d(p',q') \leq \kappa(\rho, 1, 0)$. Since $Y$ is $\CAT$ and $\beta$ and $\beta'$ are asymptotically equivalent, the distance function $d(\beta(t), \beta'(t))$ is convex and bounded. We conclude that it is bounded by $d(o,o')$. In particular, we find $q \in \beta$ with $d(o, q) = d(o', q')$ and $d(q, q') \leq d(o,o')$. Consider a geodesic $\delta$ from $o$ to $o'$. The concatenation $\delta * \alpha'$ is a $(1, d(o,o'))$-quasi-geodesic representative of $\xi$, based at $o$. By Corollary 4.3 in \cite{CashenMackay}, every point in $\delta * \alpha'$ is $\kappa'(\rho, 1, d(o,o'))$-close to $\alpha$. In particular, there exists a point $p \in \alpha$ such that $d(p, p') \leq \kappa'(\rho, 1, d(o,o'))$ and

\begin{equation*}
\begin{split}
d(o, p) & \geq d(o', p') - \kappa'(\rho, 1, d(o,o')) - d(o, o')\\
& \geq R(\rho, 1, 0, R, \kappa'(\rho, 1, d(o,o')) + \kappa(\rho, 1, 0) + d(o,o')).
\end{split}
\end{equation*}

Overall, we find that $d(p,q) \leq \kappa'(\rho, 1, d(o,o')) + \kappa(\rho, 1, 0) + d(o,o')$ and $d(o, p) \geq R(\rho, 1, 0, R, \kappa'(\rho, 1, d(o,o')) + \kappa(\rho, 1, 0) + d(o,o'))$. Using Corollary \ref{cor:Reformulation} with $J = \kappa'(\rho, 1, d(o,o')) + \kappa(\rho, 1, 0) + d(o,o')$, we conclude that $\beta$ is fellow-traveling along $\alpha$ for distance at least $R$. Since $\beta$ was an arbitrary geodesic representative of $\eta \in U_{o',R'}(\xi)$ based at $o$, we conclude that $U_{o',R'}(\xi) \subset U_{o, R}(\xi)$. By symmetry, this implies that both neighbourhood-bases induce the same topology $\FG$.

\end{proof}

\begin{rem}
Note that the assumption that $Y$ is $\CAT$ is only used to find an upper bound for $d_{Haus}(\beta, \beta')$ that does not depend on the contracting functions of $\beta$ and $\beta'$. This is a weaker assumption than being $\CAT$. For example, it is sufficient, if the distance function $d(\beta(t), \beta'(t))$ is convex for any two geodesics $\beta$, $\beta'$ in $Y$. Spaces with this convexity property are also called \emph{Busemann spaces}. Since we will focus on $\CAT$ spaces in the rest of the paper, we will not discuss this more general situation.
\end{rem}

Note that $\FG$ is not base point invariant in general. What follows is a counter example.

\begin{exam} \label{exam:counterexample}
Let $R$ be a geodesic ray (i.\,e.\,a copy of $[0, \infty)$). For $i \geq 1$, attach geodesic rays $R_i$ to the points $i \in R$. We will distinguish between the real number $x$ in $R$ and the same real number in $R_i$ by writing $R(x)$ and $R_i(x)$ respectively. Further, we denote $o := R(0)$. Let $f$ be a superlinear function, i.\,e. $\lim_{r \rightarrow \infty} \frac{r}{f(r)} = 0$, such that $f$ is injective and $f(i) > i$ for all $i \in \mathbb{N}$. Glue intervals of length $1 + i + f(i)$ into our space by attaching their endpoints to $o$ and $R_i(f(i))$. Denote these intervals by $S_i$. For $0 \leq i \leq j$, denote the interval of length $i$ in $S_j$ that start at $o$ by $I_{i,j}$. For $i$ fixed, we glue together $S_{i,j}$ for all $j \geq i$. We denote this space by $Y$ and denote the point in all $S_i$ of distance $1$ from $o$ by $o'$ (all these points have been identified by the gluing). Since $f$ is superlinear, $R$ is a $\rho$-contracting geodesic ray in $Y$ for some sublinear function $\rho$. Further, all $R_i$ are contracting as well. We write $\kappa := \kappa(\rho, 1, 0)$.

Consider the neighbourhoods of $[R] \in \partial_M Y$ in $\FG$ with respect to $o$ and $o'$. With respect to $o$, the (unique) geodesic representative of $[R_i]$ stays $\kappa$-close to $R$ for $i + \kappa$. Thus, $U_{o, r}([R])$ contains infinitely many $[R_i]$.

With respect to $o'$, however, the geodesic representative of $[R_i]$ stays $\kappa$-close to $R$ only for distance $\kappa$. Thus, $U_{o, r}([R]) = \{ [R] \}$ for $r > \kappa$. We conclude that $\FG$ is not base point independent.\\

One may ask whether $\FG$ is base point invariant when the space $Y$ admits a cocompact group action by isometries. It turns out that this is not the case either. The geometry displayed above can be embedded into a finitely generated, infinitely presented small cancellation group. This is done as follows: The space $Y$ above can be equipped with the structure of a graph all whose edges have length $1$. From now on, we consider $Y$ equipped with this graph structure. We can label the edges of this graph with letters of the alphabet $S := \{ a, b_1, b_2, b_3, b_4, b_5, b_6, c, d_1, d_2, d_3, d_4, d_5, d_6 \}$ such that:

\begin{enumerate}

\item the geodesic ray $R$ spells the word $a^{\infty}$,

\item the geodesic ray $R_i$ spells $b_1^{f(i)} b_2^{f(i)} b_3^{f(i)} b_4^{f(i)} b_5^{f(i)} b_6^{f(i)} c^{\infty}$,

\item the segment $S_i$ starting at $o$ spells $c^{i} b_1 d_6^{f(i)} d_5^{f(i)} d_4^{f(i)} d_3^{f(i)} d_2^{f(i)} d_1^{f(i)}$.

\end{enumerate}

This labeling allows us to embed the graph $Y$ into the Cayley-graph Cay$(G, S)$ of the group 

\begin{equation*}
\begin{split}
G := < & a, b_1, b_2, b_3, b_4, b_5, b_6, c, d_1, d_2, d_3, d_4, d_5, d_6\\
& \vert i \in \mathbb{N}, a^{i} b_1^{f(i)} b_2^{f(i)} b_3^{f(i)} b_4^{f(i)} b_5^{f(i)} b_6^{f(i)} d_1^{-f(i)} d_2^{-f(i)} d_3^{-f(i)} d_4^{-f(i)} d_5^{-f(i)} d_6^{-f(i)} b_1^{-1} c^{-i} >
\end{split}
\end{equation*}

generated by the labeled graph $Y$.

It is easy to see that $Y$ satisfies $Gr'(\frac{1}{6})$ (cf. \cite{Gruber}), because $f$ is injective and $f(i) > i$, and $G$ is a small cancellation group. By work of Ollivier and Gruber in \cite{Ollivier} and \cite{Gruber}, the graph $Y$ embeds isometrically into Cay$(G, S)$ (Theorem 5.10 in \cite{Gruber}). We conclude that the geodesic rays in the example above remain geodesic after being embedded into Cay$(G, S)$. Furthermore, by Theorem 4.1 of \cite{ArzhantsevaCashenGruberHume}, the geodesic rays $R$ and $R_i$ are contracting in Cay$(G,S)$ (because $f$ is superlinear). We conclude that $\FG$ on $\partial_M \mathrm{Cay}(G, S)$ exhibits the same dependence on base points as $Y$.

\end{exam}

%---------------------------------------------------------------------------------------------------------------------------------------

%CAT(0) CUBE COMPLEXES

%---------------------------------------------------------------------------------------------------------------------------------------

\subsection{$\CAT$ cube complexes} \label{subsec:CCCs}

For an in-depth introduction to $\CAT$ cube complexes, see \cite{Sageev}. We will focus on fixing notation and recalling some definitions and facts that will be used in the remainder of the paper.

Let $Y$ be a simply connected cube complex satisfying Gromov's no-$\triangle$-condition; see 4.2.C in \cite{Gromov-HypGps} and Chapter~II.5 in \cite{BH}. Unless specified otherwise, all points $v\in Y$ are implicitly understood to be vertices and all subsets of $Y$ are contained in the 0-skeleton; this applies in particular to edges and cubes. The \emph{link} of $v$ (denoted $\lk(v)$) is the simplicial complex consisting of an $(n-1)$-simplex for every $n$-cube of $Y$ based at $v$, with the same face relations. We denote by $\deg(v)$ the number of edges of $Y$ that contain $v$. 

The Euclidean metrics on the cubes of $Y$ fit together to yield a $\CAT$ metric on $Y$. We can however also endow each cube $[0,1]^k\cu Y$ with the restriction of the $\ell^1$ metric of $\R^k$ and consider the induced path metric $d_{\ell^1}(-,-)$. We refer to $d$ as the \emph{combinatorial metric} (or \emph{$\ell^1$ metric}). In finite dimensional cube complexes, the $\CAT$ and combinatorial metrics are bi-Lipschitz equivalent and complete. In particular, if all cubes in a cube complex have dimension $\leq n$, then the $\CAT$ and combinatorial metrics are $\sqrt{n}$-bi-Lipschitz equivalent.

The combinatorial metric allows us to introduce \emph{combinatorial geodesics}, which are geodesics between vertices of $Y$ with respect to the combinatorial metric that are fully contained in the $1$-skeleton of $Y$. If all cubes in $Y$ have dimension at most $n$, every combinatorial geodesic is a $(\sqrt{n}, 0)$-quasi-geodesic. Combinatorial geodesics are fully determined by the sequence of hyperplanes they cross.

We will refer to simply connected cube complexes satisfying Gromov's no-$\triangle$-condition as \emph{$\CAT$ cube complexes}, regardless of whether we consider it equipped with the $\CAT$ metric or the combinatorial metric. In order to distinguish geodesics in the $\CAT$ metric from the $\ell^1$-metric, we will call the former geodesics, while the latter will only appear in the form of the combinatorial geodesics introduced above.

We call a $\CAT$ cube complex $Y$ \emph{locally finite}, if for every vertex $v \in Y$, there are at most finitely many edges incident to $v$. We say $Y$ is \emph{uniformly locally finite} if there exists a constant $\nu$, such that for every vertex $v \in Y$, there are at most $\nu$ many edges incident to $v$. We say that $\nu$ is an upper bound for the valence of all vertices in $Y$. We say that a $\CAT$ cube complex has \emph{uniformly bounded dimension}, if there exists a constant $B$ such that every cube in $Y$ has dimension at most $B$.

Let $\mathcal{W}(Y)$ and $\mathcal{H}(Y)$ be the set of all hyperplanes and of all halfspaces of $Y$ respectively. Given a halfspace $s$, denote the other halfspace bounded by the same hyperplane by $s^*$. Given a hyperplane $h$, we call a choice of halfspace bounded by $h$ an \emph{orientation of $h$}. We sometimes denote the two orientations of $h$ by $\{ h^{+}, h^{-} \}$. The set $\mathcal{H}$ is endowed with the order relation given by inclusions; the involution $*$ is order reversing. The triple $(\mathcal{H},\cu,*)$ is thus a \emph{pocset} (see \cite{Sageev}).

Two hyperplanes are called \emph{transverse} if they intersect. Note that every intersection $h_1\cap \dots \cap h_k$ of pairwise transverse hyperplanes $h_1, \dots, h_k$ inherits a $\CAT$ cube complex structure. Its cells are precisely the intersections $h_1\cap \dots \cap h_k\cap c$ for any cube $c\cu Y$. Alternatively, $h_1\cap \dots \cap h_k$ can be viewed as a subcomplex of the cubical subdivision of $Y$.

Analogously, two halfspaces $s$, $s' \in \mathcal{H}$ are called \emph{transverse}, if their bounding halfspaces are transverse. Equivalently, they are transverse if and only if the four intersections $s \cap s'$, $s \cap s'^*$, $s^* \cap s'$, $s^* \cap s'^*$ are non-empty. A hyperplane is \emph{transverse} to a halfspace $s$ if it is transverse to the hyperplane that bounds $s$.

Given two subsets $U$, $V \subset Y$, we define $\mathcal{W}(U, V)$ to be the set of all hyperplanes that separate $U$ from $V$. Given a (combinatorial) geodesic $\gamma$, we write $\mathcal{W}(\gamma)$ for the set of hyperplanes crossed by $\gamma$.

Given a path $\gamma$ in $Y$, the length of $\gamma$ with respect to either metric can be estimated from below by $\# \mathcal{W}(\gamma) - 1$.

Let $e$ be an edge in $Y$. We denote the hyperplane crossed by $e$ by $h(e)$. We say that a hyperplane $h$ is \emph{adjacent} to a vertex $v \in Y$ if $h = h(e)$ for an edge incident to $v$.

We say that $Y$ has \emph{no free vertices} if for all vertices $v$ and all edges $e$ incident to $v$, there exists an edge $e'$ incident to $v$, such that $h(e) \cap h(e') = \emptyset$.

Suppose, $Y$ has no free vertices. Given an oriented edge $e$ we can extend it to a geodesic segment in $Y$ as follows: The orientation provides us with an endpoint $v$ of $e$. Choose an edge $e'$ incident to $v$ such that $h(e') \cap h(e) = \emptyset$. The concatenation $e * e'$, where we interpret $e$ and $e'$ as paths with arc-length parametrization, is a geodesic in $Y$. We say that we \emph{extend $e$ by $e'$}. Given a concatenation of edges $e_1 * \dots * e_n$ such that $h(e_i) \cap h(e_j) = \emptyset$ for all $i \neq j$, we can extend it to to a geodesic of length $n+1$ by adding an edge $e_{n+1}$ incident to the endpoint of $e_1 * \dots * e_n$ that does not intersect $h(e_i)$ for any $1 \leq i \leq n$. This way, we can extend any oriented edge to a geodesic of length $n+1$ by choosing a sequence of edges $e_1, \dots, e_n$, which are mutually disjoint and do not intersect with $e$. We say that we extend the oriented edge (or path) $e$ by $n$ many steps to the geodesic $e * e_1 * \dots * e_n$. Note that -- in general -- we may have several choices to extend $e$ by $n$ many steps. If $e$ is contained in a flat $F$, i.\,e. an embedding of $\mathbb{R}^2$ with standard cubulation that respects the cube complex structure, we have a unique way to extend the oriented edge $e$ by $n$ many steps inside the flat $F$ (since there is always a unique edge incident to the endpoint of the path that does not intersect any of the preceding hyperplanes). We say that we \emph{extend $e$ by $n$ many steps in the flat $F$}.

Given $y\in Y$, we denote by $\sigma_y \cu \mathcal{H}$ the set of all halfspaces containing the point $y$. It satisfies the following properties:
\begin{enumerate}
\item given any two halfspaces $s$, $s' \in \sigma_y$, we have $s \cap s' \neq \emptyset$;
\item for any hyperplane $h \in \mathcal{W}$, a side of $h$ lies in $\sigma_y$;
\item every descending chain of halfspaces in $\sigma_y$ is finite. 
\end{enumerate}
Subsets $\sigma \cu \mathcal{H}$ satisfying $(1)$--$(3)$ are known as \emph{DCC ultrafilters}. We refer to a set $\sigma \cu \mathcal{H}$ satisfying only $(1)$ and $(2)$ simply as an \emph{ultrafilter}. We will also think of an ultrafilter as a map that associates to every hyperplane one of its orientations.

Let $\iota : X \rightarrow \prod_{h \in \mathcal{W}} \{ h^{+}, h^{-} \}$ denote the map that takes each point $y$ to the set $\sigma_y$. Its image $\iota (Y)$ is precisely the collection of all DCC ultrafilters. Endowing $\prod_{h \in \mathcal{W}} \{ h^{+}, h^{-} \}$ with the product topology, we can consider the closure $\overline{\iota(Y)}$, which coincides with the set of all ultrafilters. Equipped with the subspace topology, this is a compact Hausdorff space known as the \emph{Roller compactification} of $Y$; we denote it by $\overline Y$. 

The \emph{Roller boundary} $\partial_R Y$ is defined as the difference $\overline Y\setminus \iota(Y)$. If $Y$ is locally finite, $\iota(Y)$ is open in $\overline Y$ and $\partial_R Y$ is compact; however, this is not true in general. We prefer to imagine $\partial_R Y$ as a set of points at infinity, represented by combinatorial geodesics in $Y$, rather than a set of ultrafilters. We will therefore write $y \in \partial_R Y$ for points in the Roller boundary and employ the notation $\sigma_y \cu \mathcal{H}$ to refer to the ultrafilter representing $y$. We say that a hyperplane $h$ \emph{separates} two points $x$ and $y$ in the Roller boundary if $\sigma_x$ and $\sigma_y$ do not contain the same halfspace bounded by $h$. In other words, they induce opposite orientation on $h$.

Given two points $x$, $y \in \partial_R Y$, we say they lie in the same \emph{component} if and only if there are only finitely many hyperplanes that separate $x$ from $y$. This defines an equivalence relation on $\partial_R Y$ and partitions the Roller boundary into equivalence classes, called components. Each component inherits the structure of a $\CAT$ cube complex whose hyperplanes are a strict subset of the set of hyperplanes of $Y$. We say that a hyperplane $k \in \mathcal{W}(Y)$ \emph{intersects} a component $C$ whenever it corresponds to a hyperplane in $C$. Note that for any two hyperplanes $h$, $k$ that intersect a component $C$, there exist infinitely many $h_i \in \mathcal{W}(Y)$ that intersect both $h$ and $k$.

Let $Y$ have uniformly bounded dimension. A point $x \in \partial_R Y$ is called \emph{Morse} if it admits a combinatorial geodesic representative that is Morse. Denote the set of all Morse points in the Roller boundary of $Y$ by $\partial_{R, M} Y$. By \cite{BeyFio}, there exists a surjective map $\Phi : \partial_{R, M} Y \rightarrow \partial_{M} Y$, which sends any combinatorial geodesic representative $[\gamma] \in \partial_{R, M} Y$ to the point $[\gamma]$ in the Morse boundary represented by the quasi-geodesic $\gamma$.

Let $n \in \mathbb{N}_0$. We call two hyperplanes $h$, $h'$ \emph{$n$-strongly separated}, if they are disjoint and there are at most $n$ many hyperplanes that intersect both $h$ and $h'$. For $n = 0$ we simply write \emph{strongly separated}.

In \cite{CharneySultan}, Charney and Sultan characterized Morse geodesics in uniformly locally finite $\CAT$ cube complexes as follows:

\begin{thm}[Theorem 4.2 in \cite{CharneySultan}] \label{thm:CharneySultan}
Let $Y$ be a uniformly locally finite $\CAT$ cube complex. There exist $r > 0$, $n \geq 0$ (depending only on $D$ and the maximal valence $\nu$), such that a geodesic ray $\gamma$ in $Y$ is $D$-contracting if and only if $\gamma$ crosses an infinite sequence of hyperplanes $h_1, h_2 \dots$ at points $y_i := \gamma \cap h_i$ satisfying

\begin{enumerate}
\item $h_i$, $h_{i+1}$ are $n$-separated and
\item $d(y_i, y_{i+1}) < r$.
\end{enumerate}
\end{thm}

\begin{rem}
Consider a Morse geodesic ray $\gamma$ in a uniformly locally finite $\CAT$ cube complex $Y$. This geodesic ray induces a consistent orientation of hyperplanes and thus a point $y \in \partial_R Y$. The theorem by Charney-Sultan implies that the component $C \subset \partial_R Y$ that contains $y$ is a bounded cube complex. If it were unbounded, we would find for any $N > 0$ a time $T$, such that any two hyperplanes crossed by $\gamma \vert_{[T, \infty)}$, are not $N$-strongly separated (no hyperplane crossed by $\gamma$ can intersect $C$, as $\gamma$ is a geodesic and moves infinitely far away from any hyperplane it crosses). We conclude that the preimage $\Phi^{-1}([\gamma])$ of Morse points $[\gamma] \in \partial_M Y$ is always a bounded component.
\end{rem}

%---------------------------------------------------------------------------------------------------------------------------------------

%RIGHT-ANGLED ARTIN GROUPS

%---------------------------------------------------------------------------------------------------------------------------------------

\subsection{Right-angled Artin groups} \label{subsec:RAAGs}

Let $\Gamma$ be a finite, undirected graph with no multiple edges and no loops of length 1. We denote its vertices by $V_{\Gamma}$ and its edges by $E_{\Gamma}$. The Right-angled Artin group (RAAG) associated to $\Gamma$ is defined by

\[ A_{\Gamma} := \langle V_{\Gamma} | [v_i, v_j] \text{ for all } (v_i, v_j) \in E_{\Gamma} \rangle \]

We now build a cube complex whose fundamental group is $A_{\Gamma}$. Its universal covering $Y_{\Gamma}$ will inherit a cube complex structure and will be $\CAT$. Furthermore, $A_{\Gamma}$ acts properly, cocompactly by cube-automorphisms on $Y_{\Gamma}$. We start with one vertex and glue both endpoints of $V_{\Gamma}$-many edges to this vertex. We label the edges by the vertices $v_i \in V_{\Gamma}$. Whenever we have vertices $v_{i_1}, \dots, v_{i_k}$ such that $(v_{i_j}, v_{i_j'}) \in E_{\Gamma}$, we glue a $k$-cube along the edges $v_{i_1}, \dots v_{i_k}$ in such a way that the $k$-cube is glued to become a $k$-dimensional torus (i.\,e. parallel edges of the $k$-cube are all glued to the same edge in the complex). The resulting cube complex is called the Salvetti complex of $\Gamma$. We denote its universal covering by $Y_{\Gamma}$. The following statement is proven during the proof of Theorem 5.1 in \cite{CordesHume}.

\begin{prop}[\cite{CordesHume}]
Given a finite, undirected graph $\Gamma$, any contracting geodesic ray in $Y_{\Gamma}$ crosses an infinite sequence of hyperplanes as in Theorem \ref{thm:CharneySultan}, where all hyperplanes in the sequence are strongly separated.
\end{prop}

\begin{rem}
In \cite{Fernos,Fernos-Lecureux-Matheus}, the notion of a regular point was introduced, which can be defined as follows (cf. Proposition~7.5 in \cite{Fernos}): A point $\xi \in \partial_R Y$ is called \emph{regular} if the ultrafilter $\sigma_{\xi}$ contains an infinite chain $h_0 \subsetneq h_1 \subsetneq \dots$ such that the corresponding hyperplanes $w_0, w_1 \dots$ are strongly separated. From the results above, it follows that the Morse boundary of a RAAG is contained in the image of the regular points under the map $\Phi$.
\end{rem}

%---------------------------------------------------------------------------------------------------------------------------------------

%Comparing FG and FQ

%---------------------------------------------------------------------------------------------------------------------------------------

\section{Comparing $\FG$ and $\FQ$} \label{sec:Comparison}

Let $Y$ be a $\CAT$ cube complex and fix a base point $o \in Y$ for the rest of this section. By Cashen-Mackay, the contracting boundary $\partial_M Y$ carries the topologies of fellow-traveling geodesics -- $\mathcal{FG}$ -- and of fellow-traveling quasi-geodesics -- $\mathcal{FQ}$. We will introduce a third topology and then show that it coincides with $\mathcal{FG}$ but differs from $\mathcal{FQ}$ in general.

Let $h_1, \dots, h_n$ be distinct hyperplanes in $Y$. Define the set

\begin{equation*}
\begin{split}
U_{o, h_1, \dots, h_n} := \{ \xi \in \partial_M Y | \text{The unique geodesic} & \text{ representative of $\xi$ based at $o$}\\
& \text{ crosses the hyperplanes $h_1, \dots, h_n$.} \}.
\end{split}
\end{equation*}

It is easy to see that the collection $\{ U_{o, h_1, \dots, h_n} \}_{n, h_1, \dots, h_n}$ is a basis of the topology $\Hyp$ introduced in section \ref{sec:Introduction}.

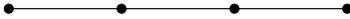
\begin{figure} 
\begin{tikzpicture}[scale=1.50]
\draw [fill] (-1.5,0) circle [radius=0.04cm];
\draw [fill] (-1.5,0) -- (-0.5, 0);
\draw [fill] (-0.5, 0) circle [radius=0.04cm];
\draw [fill] (-0.5, 0) -- (0.5, 0);
\draw [fill] (0.5, 0) circle [radius=0.04];
\draw [fill] (0.5, 0) -- (1.5, 0);
\draw [fill] (1.5, 0) circle [radius=0.04];
\end{tikzpicture}
\caption{The graph $\Gamma_0$ that induces $G_0 := A_{\Gamma_0}$.}
\label{fig:Graph} 
\end{figure}

\begin{figure}
  \centering
  \def\svgwidth{3in}
  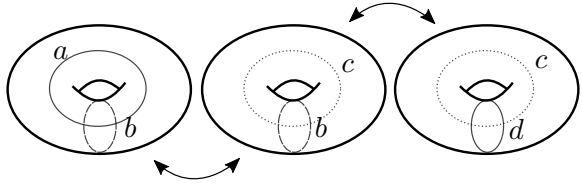
\caption{The Croke--Kleiner example}
\label{fig:Torus}
\end{figure}

\begin{prop} \label{prop:geodesic topology}
Let $Y$ be a uniformly locally finite $\CAT$ cube complex, $X$ its contracting boundary as a set. Then $\FG = \Hyp$ on $X$.%We need uniform local finiteness only to use the Theorem by Charney-Sultan. For the rest, a uniform bounded dimension is sufficient.
\end{prop}

\begin{cor} \label{cor:basepointindependenceofHYP}
$\Hyp$ is independent of the choice of base point.
\end{cor}

Consider the right-angled Artin group $G_0 := A_{\Gamma_0}$ induced by the graph $\Gamma_0$ depicted in figure \ref{fig:Graph}. Its Salvetti complex can be obtained as follows. Consider three distinct tori and consider two simple closed curves in each as depicted in figure \ref{fig:Torus}. Denoting the curves of the $i$-th torus by $a_i$ and $b_i$, we obtain the Salvetti complex of $G$ by glueing the tori by identifying $b_1$ with $a_2$ and $b_2$ with $a_3$. Denote its universal covering by $Y_0 := Y_{\Gamma_0}$.

\begin{prop} \label{prop:quasi-geodesic topology}
If $X = \partial_M Y_0$ for $\Gamma_0$ as in figure \ref{fig:Graph}, then $\Hyp \subsetneq \FQ$ on $X$.
\end{prop}

\begin{proof}[Proof of Corollary \ref{cor:basepointindependenceofHYP}]
This follows from Proposition \ref{prop:basepointindependence} and Proposition \ref{prop:geodesic topology}
\end{proof}

\begin{proof} [Proof of Proposition \ref{prop:geodesic topology}]
Suppose, all cubes in $Y$ have dimension at most $n$. We start by showing that $\FG \subset \Hyp.$ Let $\xi \in X$ and $U_{o, R}(\xi)$ an element of the neighbourhood basis of $\FG$. Denote the geodesic representative of $\xi$ starting at $o$ by $\gamma$. Let $(h_i)_i$ be the sequence of hyperplanes crossed by $\gamma$ ordered in the order they are crossed by $\gamma$. We find a constant $D_{\xi}$ such that the geodesic representative of $\xi$ starting at $o$ is $D_{\xi}$-contracting. Denote $\kappa := \kappa(D_{\xi}, 1, 0)$. By \cite{CordesHume}, we find $n \in \mathbb{N}$, $r > 0$ and a subsequence $(h_{i_j})_j$ of $(h_i)_i$ consisting of pairwise disjoint, $m$-strongly separated hyperplanes such that $d(h_{i_j}, h_{i_{j+1}}) < r$. Let $h_{i_N}$ be the first hyperplane in $(h_{i_j})_j$ such that $N > 1$ and $d(o, h_{i_j}) > R + C$, where
\[ C := 4\sqrt{n}(m + \sqrt{n}r) + 8\kappa.\]

Consider the basis-element $U_{o, h_{i_{N+2}}}$. Clearly, $\xi \in U_{o, h_{i_{N+2}}}$. We claim that $U_{o, h_{i_{N+2}}} \subset U_{o, R}(\xi)$. Let $\eta \in U_{o, h_{i_{N+2}}}$ and let $\beta$ be the geodesic representative of $\eta$ starting at $o$. Note that, since $\beta$ is a geodesic, it can cross every hyperplane at most once. Since $\eta \in U_{o, h_{i_{N+2}}}$, $\beta$ has to cross $h_{i_{N}}$, $h_{i_{N+1}}$ and $h_{i_{N+2}}$ at some point.

Denote the restrictions of $\alpha$ and $\beta$ to the area between $h_{i_j}$ and $h_{i_{j+1}}$ by $\alpha_j$ and $\beta_j$ respectively. We focus our attention on the restriction of $\alpha$ and $\beta$ to the area between $h_{i_N}$ and $k_{i_{N+2}}$, i.\,e. on $\alpha_N$, $\alpha_{N+1}$, $\beta_N$ and $\beta_{N+1}$. We will estimate the distance between $p := \alpha \cap h_{i_{N+1}}$ and $q := \beta \cap h_{i_{N+1}}$. Since $Y$ has uniformly bounded dimension, we can do so by estimating the number of hyperplanes that separate $p$ from $q$.

Suppose, $k \in \mathcal{W}(p,q)$. Since $\alpha$ and $\beta$ both start at $o$, $k$ has to be crossed by exactly one of them at some point before the two paths cross $h_{i_{N+1}}$. There are three possibilities.

Case 1: Suppose, $k$ intersects neither $\alpha_N$ nor $\beta_N$. Then, $k$ intersects $h_{i_N}$. Since $h_{i_N}$ and $h_{i_{N+1}}$, there can be at most $m$-many such hyperplanes.

Case 2: Suppose, $k$ intersects $\alpha_N$. By assumption, $d(\alpha \cap h_{i_N}, \alpha \cap h_{i_{N+1}}) < r$. Therefore, there can be at most $\sqrt{n}r$-many such hyperplanes.

Case 3: Suppose, $k$ intersects $\beta_N$. Focus on the area between $h_{i_{N+1}}$ and $h_{i_{N+2}}$. Since $\beta$ is a geodesic, $k$ has to intersect either $\alpha_{N+1}$ or $h_{i_{N+2}}$. Again, there are at most $m + \sqrt{n}r$-many such hyperplanes.

We conclude that, in total, there are at most $2(m+ \sqrt{n}r)$-many hyperplanes that separate $p$ from $q$. Thus, $d(p,q) \leq 2\sqrt{n}(m + \sqrt{n}r)$. Lemma 4.6 in \cite{CashenMackay} now implies that, there exists a point $q' \in \beta$ which is $\kappa$-close to $\alpha \diagdown B_R(o)$. This relies on our choice of $C$ which was made specifically to suit Cashen-Mackays result. This implies that $\eta \in U_{o, R}(\xi)$. Therefore, $\xi \in U_{o, k_{i_N}} \subset U_{o, R}(\xi)$ which implies that $U_{o, R}(\xi) \in \Hyp$.\\

We are left to show that $\Hyp \subset \FG$. Let $h_1, \dots, h_n$ be distinct hyperplanes and $\xi \in U_{o, h_1, \dots, h_n}$. Let $\gamma$ be the geodesic representative of $\xi$ based at $o$. We find a constant $D_{\xi}$ such that $\gamma$ is $D_{\xi}$-contracting. Choose $R > 0$ sufficiently large, such that $d(h_i, \gamma \diagdown B_R(o)) > \kappa(D_{\xi}, 1, 0)$ for all $i$. Such $R$ exists, since a geodesic $\gamma$ cannot stay uniformly close to any hyperplane it crosses (for example because the distance function $d(\gamma(t), \beta(t))$ of two geodesics in a $\CAT$ space is convex). Clearly, $\xi \in U_{o, R}(\xi)$ and we need to show that $U_{o, R}(\xi) \subset U_{o, h_1, \dots, h_n}$.

Let $\eta \in U_{o, R}(\xi)$. Then there exists a point $p \in \beta$, such that $d(p, \gamma \diagdown B_R(o)) \leq \kappa(D_{\xi}, 1, 0)$. In particular, if $\gamma_s \in \gamma \diagdown B_R(o)$ satisfies $d(p, \gamma) = d(p, \gamma_s)$, then the geodesic $\delta$ from $p$ to $\gamma_s$ is completely contained in the $\kappa(D_{\xi}, 1, 0)$-neighbourhood of $\gamma \diagdown B_R(o)$. Since $d(h_i, \gamma \diagdown B_R(o)) > \kappa(D_{\xi}, 1, 0)$, $\delta$ cannot intersect $h_i$ for any $i$. We conclude that, for all $i$, $h_i$ separates $o$ from $p$, which implies that $\beta$ crosses $h_i$ for all $i$. Therefore, $\eta \in U_{o, h_1, \dots, h_n}$ and $\xi \in U_{o, R}(\xi) \subset U_{o, h_1, \dots, h_n}$, which completes the proof.
\end{proof}

\begin{proof} [Proof of Proposition \ref{prop:quasi-geodesic topology}]
By \cite{CashenMackay}, we know that $\FG \subset \FQ$, so we are left to show that $\Hyp \neq \FQ$. In particular, we need to show that there exists some neighbourhood $U_{o, R}(\xi)$ such that there exists no $h_1, \dots, h_n$ such that $\xi \in U_{o, h_1, \dots, h_n} \subset U_{o, R}(\xi)$.\\

We recall a few facts about the space $Y_0$. In \cite{CrokeKleiner}, Croke and Kleiner showed that $Y_0$ can be written as a union of the following `blocks': Consider two of the tori in the Salvetti complex that are identified along a curve (i.\,e.\,the first and second torus, or the second and third torus). The preimage of this subspace under the projection from the universal covering is a disjoint union with each component being isometric to the product of a $4$-valent tree with $\mathbb{R}$. Furthermore, the preimage of any of the three tori is a disjoint union of $2$-dimensional flats. We denote the collection of flats corresponding to the first torus by $\mathcal{A}$, the collection corresponding to the second torus by $\mathcal{B}$ and the collection corresponding to the third torus by $\mathcal{C}$. The Salvetti complex of $\Gamma$ has four hyperplanes. Denote the collection of preimages of the hyperplane crossed by the curve $a_1$ by $A$, the collection of preimages of the hyperplane crossed by the curve $b_1 = a_2$ by $B$, the collection of preimages of the hyperplane crossed by the curve $b_2 = a_3$ by $C$ and the collection of preimages of the hyperplane crossed by the curve $b_3$ by $D$. Thus, the collection of hyperplanes in $Y_0$ is the disjoint union of $A$, $B$, $C$, $D$.\\

Before we start with the main construction, we remark that, whenever we concatenate paths in the following proof, we do not rescale their parametrisation. We adjust their parametrisation by time-shift to allow concatenation, but they remain arc-length parametrised.\\

We will introduce a specific point $\xi \in \partial_M Y_0$ and find some neighbourhood $U_{o, R}(\xi)$ as described above. Fix a base point $o \in Y_0^{(0)}$ and choose edges, incident to $o$, that induce hyperplanes $b_0 \in B$ and $c_1 \in C$ respectively. Consider the unique geodesic segment that starts at $o$, has length $\sqrt{2}$ and crosses both $b_0$ and $c_1$ at an angle of $\frac{\pi}{4}$. Denote this geodesic segment by $\gamma_1$ and its second endpoint by $p_1$. Note that both endpoints of $\gamma_1$ are vertices in $Y_0$. Let $c_1' \in C$ be the unique hyperplane in $C$ that has minimal distance to $p_1$ and choose a hyperplane $d_1 \in D$ of minimal distance to $p_1$. Define $\gamma_2$ to be the unique geodesic segment that starts at $p_1$, has length $\sqrt{2}$ and crosses both $c_1'$ and $d_1$ at an angle of $\frac{\pi}{4}$. Denote the second endpoint of $\gamma_2$ by $p_2$. Let $c_1'' \in C$ be the unique hyperplane in $C$ that is not equal to $c_1'$ and has minimal distance from $p_2$. Choose $b_1 \in B$ with minimal distance from $p_2$. Define $\gamma_3$ to be the unique geodesic segment that starts at $p_2$, has length $\sqrt{2}$ and crosses both $c_1''$ and $b_1$ at an angle of $\frac{\pi}{4}$. Denote the second endpoint of $\gamma_3$ by $p_3$. Let $b_1'$ the unique hyperplane in $B$ that is not equal to $b_1$ and has minimal distance from $p_3$. Choose $a_1 \in A$ with minimal distance to $p_3$. Define $\gamma_4$ to be the unique geodesic segment that starts at $p_3$, has length $\sqrt{2}$ and crosses both $b_1'$ and $a_1$ at an angle of $\frac{\pi}{4}$. Denote the second endpoint of $\gamma_4$ by $p_4$. Let $b_1''$ be the unique hyperplane in $B$ not equal to $b_1'$ and has minimal distance from $p_4$. Choose $c_2 \in C$ with minimal distance to $p_4$. Repeat the construction above to construct a contracting geodesic $\gamma$ in $Y_0$, which is the concatenation of all the $\gamma_i$, and denote $\xi := [\gamma] \in \partial_M Y_0$. Note that $\gamma$ crosses at most three hyperplanes in every flat it crosses.

The image of $\gamma$ crosses an infinite collection of flats $(F_l)_l$, with $F_l \in \mathcal{A} \cup \mathcal{B} \cup \mathcal{C}$ ordered in the way they are crossed by $\gamma$. Indeed, $\gamma_l = \gamma \cap F_l$, where we identify any path segment with its image. Denote by $a_l$, $b_l$, $c_l$, $d_l$ the number of hyperplanes in $A$, $B$, $C$, $D$ respectively crossed by $\gamma_l$.\\

Let $R \gg 0$ and pick any finite collection of hyperplanes $h_1, \dots, h_n$ each of which is separating $o$ from $\xi$. Since $\gamma$ is contracting, we find a sequence of strongly separated hyperplanes $(h_{i_j})_j$ as before. In particular, we find a hyperplane $k := h_{i_N}$ such that for all $i$, $k$ separates $\xi$ from $h_i$. We have $\xi \in U_{o, k} \subset U_{o, h_1, \dots, h_n}$. We will construct a $(8\sqrt{2},1)$-quasi-geodesic $\beta_k'$ such that $[\beta_k'] \in U_{o, k} \diagdown U_{o, R}(\xi)$.

The main step of the proof is the construction of an auxiliary quasi-geodesic $\beta$, which will be a path in the $1$-skeleton of $Y_0$, parametrized by unit speed on every edge. In particular, it will be determined completely by the ordered sequence of hyperplanes it crosses. We will write $\beta$ as the concatenation of paths $(\beta_l)_{l \in \mathbb{N}}$, where $\beta_l \subset F_l$. In what follows, whenever we consider a path $p : [a,b] \rightarrow Y_0$, we will denote $p(a)$ to be its starting point and $p(b)$ to be its endpoint.

Choose an integer $3 < \Delta < R$. Let $b_0' \in B$ be the unique hyperplane in $B$ that is not equal to $b_0$ and has minimal distance to $o$. Consider the geodesic starting at $o$ and crossing through $b_0'$ and extend this geodesic by $\Delta + 3$ many steps in the flat $F_1$. Denote this geodesic line by $p_1$. Define $q_1$ to be the geodesic from the endpoint of $p_1$ to its closest point projection onto $F_1 \cap F_2$. Note that this geodesic is completely contained in the $1$-skeleton of $Y_0$, as the endpoint of $p_1$ is separated from $F_1 \cap F_2$ only by hyperplanes in $C$. Denote $\beta_1 := p_1 * q_1$. Note that $\beta_1$ is a $(\sqrt{2}, 0)$-quasi-geodesic because the $l^1$-metric on the euclidean plane is $\sqrt{2}$-bi-Lipschitz-equivalent to the euclidean metric on the plane. Further, the endpoint of $\beta_1$ is separated from $\gamma$ by $\Delta + 3$-many hyperplanes that pairwise don't intersect. Therefore, the distance of the endpoint of $\beta_1$ to any point in $\gamma$ is at least $\Delta + 3$.

We now provide the general definition of $\beta_l$ by an inductive procedure. Suppose, we have defined $(\sqrt{2},0)$-quasi-geodesics $\beta_k \subset F_k$ for all $k < l$. Additionally, suppose that for all $1 < k < l$, the distance of any point in $\beta_k$ to any point in $\gamma$ is at least $\Delta$. Further, assume that for all $1 \leq k < l$, the endpoint of $\beta_k$ lies in $F_{k} \cap F_{k+1}$ and has distance at least $\Delta + 3$ from any point in $\gamma$.

Denote the endpoints of $\beta_k$ and $\gamma_k$ by $v_k$ and $w_k$ respectively for all $k$. Define
\[ M_l := d(v_{l-1}, F_l \cap F_{l+1}) = \# \{ \text{hyperplanes separating }v_{l-1} \text{ from } F_l \cap F_{l+1} \}\]

and
\[N_l := \max \left(\Delta + 3, 5 M_l, 2 \sum_{k < l} l(\beta_k) \right).\]

Note that this implies
\begin{equation} \label{eq:keyforN}
\frac{N_l}{2} - M_l \geq \frac{N_l}{4} + \frac{M_l}{8}.
\end{equation}

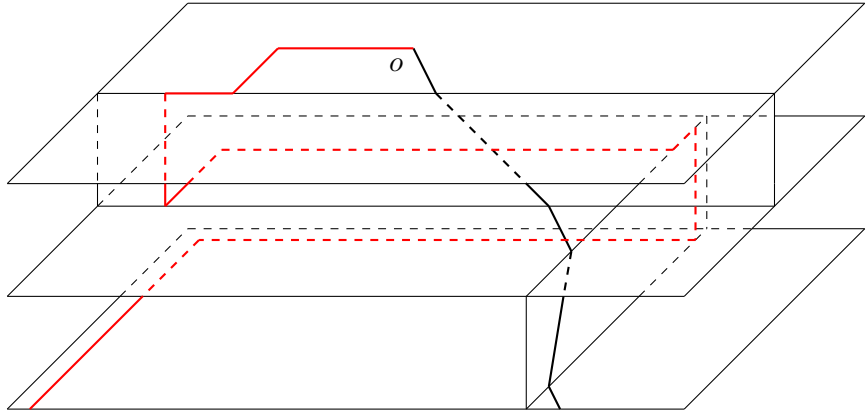
\begin{figure}
\begin{tikzpicture}[scale=1.50]
%First flat
\draw [fill] (-3,0) -- (3,0);
\draw [fill] (-3,0) -- (-1.4, 1.6);
\draw [fill] (-1.4, 1.6) -- (4.6, 1.6);
\draw [fill] (3,0) -- (4.6, 1.6);
%Second flat
\draw [fill] (-2.2, 0.8) -- (3.8, 0.8);
\draw [dashed] (-2.2, 0.8) -- (-2.2, 0);
\draw [fill] (3.8, 0.8) -- (3.8, 0);
\draw [fill] (-2.2, 0) -- (-2.2, -0.2);
\draw [fill] (3.8, 0) -- (3.8, -0.2);
\draw [fill] (-2.2, -0.2) -- (3.8, -0.2);
%Third flat
\draw [fill] (-3, -1) -- (3,-1);
\draw [fill] (-3,-1) -- (-2.2, -0.2);
\draw [dashed] (-2.2, -0.2) -- (-1.4, 0.6);
\draw [dashed] (-1.4, 0.6) -- (3.8, 0.6);
\draw [fill] (3.8, 0.6) -- (4.6, 0.6);
\draw [fill] (3,-1) -- (4.6, 0.6);
%Fourth flat
\draw [fill] (1.6, -1) -- (2.6, 0);
\draw [dashed] (2.6, 0) -- (3.2, 0.6);
\draw [fill] (1.6, -1) -- (1.6, -2);
\draw [dashed] (3.2, 0.6) -- (3.2, -0.4);
\draw [fill] (1.6, -2) -- (2.6, -1);
\draw [dashed] (2.6, -1) -- (3.2, -0.4);
%Fifth flat
\draw [dashed] (-1.4, -0.4) -- (3.6, -0.4);
\draw [fill] (3.6, -0.4) -- (4.6, -0.4);
\draw [fill] (-3, -2) -- (-2, -1);
\draw [dashed] (-2, -1) -- (-1.4, -0.4);
\draw [fill] (3, -2) -- (4.6, -0.4);
\draw [fill] (-3, -2) -- (3, -2);
% Path \gamma
\draw [fill, thick] (0.6, 1.2) -- (0.8, 0.8);
\node [below left] at (0.6, 1.2) {$o$};
\draw [dashed, thick] (0.8, 0.8) -- (1.6, 0);
\draw [fill, thick] (1.6, 0) -- (1.8, -0.2);
\draw [fill, thick] (1.8, -0.2) -- (2, -0.6);
\draw [dashed, thick] (2, -0.6) -- (1.93, -1);
\draw [fill, thick] (1.93, -1) -- (1.8, -1.8);
\draw [fill, thick] (1.8, -1.8) -- (1.9, -2);
%Path \beta
\draw [fill, thick, red] (0.6, 1.2) -- (-0.6, 1.2);
\draw [fill, thick, red] (-0.6, 1.2) -- (-1, 0.8);
\draw [fill, thick, red] (-1, 0.8) -- (-1.6, 0.8);
\draw [dashed, thick, red] (-1.6, 0.8) -- (-1.6, 0);
\draw [fill, thick, red] (-1.6, 0) -- (-1.6, -0.2);
\draw [fill, thick, red] (-1.6, -0.2) -- (-1.4, 0);
\draw [dashed, thick, red] (-1.4, 0) -- (-1.1, 0.3);
\draw [dashed, thick, red] (-1.1, 0.3) -- (2.9, 0.3);
\draw [dashed, thick, red] (2.9, 0.3) -- (3.1, 0.5);
\draw [dashed, thick, red] (3.1, 0.5) -- (3.1, -0.5);
\draw [dashed, thick, red] (3.1, -0.5) -- (-1.3, -0.5);
\draw [dashed, thick, red] (-1.3, -0.5) -- (-1.8, -1);
\draw [fill, thick, red] (-1.8, -1) -- (-2.8, -2);
\end{tikzpicture}
\caption{The concatenation of the quasi-geodesic $\beta_k$ (red) and the geodesic $\gamma$ (black). The scale of $\beta_k$ is reduced for presentation.}
\label{fig:path} 
\end{figure}

We distinguish between three cases:\\

Case 1: If $F_l \in \mathcal{A}$, then $F_{l-1} \in \mathcal{B}$. In this case, the $1$-dimensional cubically embedded line $F_{l-1} \cap F_l$ is only crossed by hyperplanes in $B$. Since $v_{l-1}$ and $w_{l-1}$ both lie in $F_{l-1} \cap F_l$, we conclude that they are only separated by hyperplanes in $B$. In particular, since the distance between them has to be at least $\Delta + 3$ by the assumption of the induction, there are at least $\Delta + 3$ many hyperplanes separating $v_{l-1}$ and $w_{l-1}$.

By construction, $\gamma_l$ crosses one hyperplane in $B$ and one hyperplane in $A$. Define $p_l$ to first cross the unique hyperplane in $B$ adjacent to $v_{l-1}$ that does not separate $v_{l-1}$ from $w_{l-1}$. Extend $p_l$ by $N_l - 1$ many steps in $F_l$. Define $q_l$ to be the path starting at the endpoint of $p_l$ and crossing the hyperplane in $A$ that is crossed by $\gamma_l$. Define $\beta_l := p_l * q_l$.

We now prove that $\beta_l$ satisfies the assumptions of the induction. Clearly, it is a $(\sqrt{2}, 0)$-quasi-geodesic in $F_l$. Further, since $\gamma$ crosses at most three hyperplane in every flat and $v_{l-1}$ and $w_{l-1}$ are separated by at least $\Delta + 3$ many hyperplanes, any point in $\beta_l$ is separated by at least $\Delta$ many hyperplanes in $B$ from any point in $\gamma$. Since no two hyperplanes in $B$ intersect, this implies that the distance between any point in $\beta_l$ and any point in $\gamma$ is at least $\Delta$. Further, the endpoint of $\beta_l$ is separated by $\Delta + 3$ many additional pairwise non-intersecting hyperplanes from any point in $\gamma$. We conclude that the distance of the endpoint of $\beta_l$ to any point in $\gamma$ is at least $\Delta + 3$.

Finally, we show that the endpoint of $\beta_l$ lies in $F_l \cap F_{l+1}$. Since $\beta_l$ starts in $F_l$ and crosses only hyperplanes that intersect with $F_l$, it lies in $F_l$. Given a vertex $v \in F_l$, it lies in $F_l \cap F_{l+1}$ if and only if there exists a point $w \in F_{l+1}$ such that no hyperplane in $A$ separates $v$ from $w$. This is the case for the endpoints $v_l$, $w_l$ of $\beta_l$ and $\gamma_l$ and we know that $w_l$ lies in $F_{l+1}$ by definition of $\gamma$ and $F_l$. We conclude that $\beta_l$ satisfies all the assumption of the induction.\\

Case 2: If $F_l \in \mathcal{C}$, then we do the same thing as in case 1, except that the roles of the sets $A$ and $B$ are played by the sets $D$ and $C$ respectively.\\

Case 3: If $F_l \in \mathcal{B}$, assume without loss of generality that $F_{l-1} \in \mathcal{A}$. Again, the case of $F_{l-1} \in \mathcal{C}$ can be obtained by role-swapping of $A$ with $D$ and of $B$ with $C$. We know that $v_{l-1}$, $w_{l-1} \in F_{l-1} \cap F_l$ and they are separated by at least $\Delta + 3$ many hyperplanes in $B$. By construction of $\gamma$, there is a unique hyperplane in $C$ that is adjacent to $w_{l-1}$ and is not crossed by $\gamma_l$. Define $p_l$ to start at $w_{l-1}$, to cross this unique hyperplane and extend it by $N_l - 1$ many steps in $F_l$. Define $q_l$ to start at the endpoint of $p_l$ and to cross all hyperplanes in $B$ that separate that endpoint from the line $F_l \cap F_{l+1}$. Note that these are all hyperplanes that separate the endpoint of $p_l$ from $F_l \cap F_{l+1}$, as all hyperplanes crossed by $p_l$ are also crossed by $F_l \cap F_{l+1}$. Define $\beta_l := p_l * q_l$.

Again, we show that $\beta_l$ satisfies all the assumption of the induction. It is a $(\sqrt{2},0)$-quasi-geodesic by construction. By the same argument as before, there are $\Delta$-many hyperplanes in $B$ that separate any point of $p_l$ from any point on $\gamma$. Furthermore, there are $N_l \geq \Delta + 3$ many hyperplanes that separate any point on $q_l$ from any point on $\gamma$, specifically, the hyperplanes crossed by $p_l$. %Thus N_l \geq \Delta + 3 is necessary.
We conclude that $\beta_l$ satisfies the assumption of the induction.\\

Note that for any $l > 1$, there is no hyperplane crossed by both $p_{l-1}$ and $p_l$. Together with the fact that $q_{l-1}$ never crosses hyperplanes of the same type as $p_{l-1}$ and $p_l$, this implies in particular that $p_{l-1} * q_{l-1} * p_l$ is a combinatorial geodesic.

\begin{claim}
The path $\beta$ is an $(8\sqrt{2},1)$-quasi-geodesic.
\end{claim}

Before we prove the claim, we point out that $\beta$ is not contracting. We will construct suitable contracting paths $\beta_k'$ from it, once we have proven the claim.

\begin{proof}[Proof of Claim]
Let $0 \leq t < s$. Since $\beta$ is a concatenation of geodesic lines, we have $d(\beta(t),\beta(s)) \leq | s - t |$. We are left to show that $d(\beta(t), \beta(s)) \geq \frac{1}{8\sqrt{2}} | s - tÊ| - 1$. Suppose, $\beta(s) \in F_l$. Then $\beta(t) \in F_k$ for some $k \leq l$. There are three cases.\\

Case 1: If $k=l$, we find $t'$, $s'$ such that $t-s = t'-s'$ and $\beta(t) = \beta_l(t')$, $\beta(s) = \beta_l(s')$. We have already seen that $\beta_l$ is a $(\sqrt{2}, 0)$-quasi-geodesic by construction.

Case 2: Suppose $k = l-1$. We first note that $\beta$ can be written as a concatenation of geodesic lines $p_l$, $q_l$ in the $1$-skeleton of $Y_0$. Each of these geodesic lines crosses hyperplanes from only one of the four families $A$, $B$, $C$, $D$. If we write down for every geodesic line, which type of hyperplanes it crosses, in order of concatenation, we get a periodic sequence with period
\[ C, B, C, D, B, C, B, A. \]
In other words, $p_1$ crosses hyperplanes in $C$, $q_1$ crosses hyperplanes in $B$, $p_2$ crosses hyperplanes in $C$, $q_2$ crosses hyperplanes in $D$, etc. following the periodic sequence described above.

We will distinguish between several subcases.

Case 2a: Suppose, $\beta(s) \in p_l$ and $\beta(t) \in q_{l-1}$. By construction, the concatenation $q_{l-1} * p_l$ is either a geodesic, or contained inside the $1$-skeleton of the flat $F_{l-1}$, hence a $(\sqrt{2},0)$-quasi-geodesic. 

Case 2b: Suppose, $\beta(s) \in p_l$ and $\beta(t) \in p_{l-1}$. As we noted before stating the claim, the concatenation $p_{l-1} * q_{l-1} * p_l$ is a combinatorial geodesic for all $l > 1$. Thus, we have

\begin{equation*}
\begin{split}
d(\beta(s), \beta(t)) & \geq \frac{1}{\sqrt{2}} d_{l^1}(\beta(s), \beta(t))\\
& = \frac{1}{\sqrt{2}} | s - t |.
\end{split}
\end{equation*}

Case 2c: Suppose, $\beta(s) \in q_l$ and $\beta(t) \in q_{l-1}$. By checking the type of hyperplanes crossed by $q_{l-1}$, $p_l$ and $q_l$, we see that $q_{l-1} * p_l * q_l$ is a combinatorial geodesic and thus a $(\sqrt{2}, 0)$-quasi-geodesic by the same argument as in Case 2b.

Case 2d: Suppose, $\beta(s) \in q_l$ and $\beta(t) \in p_{l-1}$. By definition of $N_l$, we know that $p_l$ crosses at least twice as many hyperplanes as all $\beta_k$ for $k < l$ together. We conclude that $\beta(s)$ is separated from $\beta(t)$ by at least $\frac{N_l}{2} - M_l$ many hyperplanes and $| s - t | \leq 2N_l + M_l$. Therefore, using equation (\ref{eq:keyforN}), we have%In this line, we use that N_l is at least twice as large as everything before together.

\begin{equation*}
\begin{split}
d(\beta(s), \beta(t)) & \geq \frac{1}{\sqrt{2}} d_{l^1}(\beta(s), \beta(t))\\
& \geq \frac{1}{\sqrt{2}} \left( \frac{N_l}{2} - M_l \right)\\
& \geq \frac{1}{8\sqrt{2}} \left( 2N_l + M_l \right)\\ %Here we use N_l \geq 5 M_l.
& \geq \frac{1}{8\sqrt{2}} | s - t |.
\end{split}
\end{equation*}

Combining cases 2a-d, we conclude that $p_{l-1} * q_{l-1} * p_l * q_l$ is a $(8\sqrt{2},0)$-quasi-geodesic.

Case 3: Suppose $k < l-1$. If $\beta(s) \in q_l$, the argument from Case 2d applies and thus, the $(8\sqrt{2},0)$-quasi-geodesic inequalities are satisfied.

If $\beta(s) \in p_l$, let $\beta(T)$ be the end point of $p_{l-1}$. Clearly, $\beta(s)$ and $\beta(T)$ are separated by at least $s - T - \frac{1}{2}$ many hyperplanes. Furthermore, $p_{l-1}$ provides at least $\frac{N_{l-1}}{2}$ many additional hyperplanes that separate $\beta(s)$ from any $\beta(t) \in \beta_k$ for $k < l-1$, as $p_{l-1}$, $q_{l-1}$ and $p_l$ cross mutually disjoint families of hyperplanes. Using the fact that $| T - t | \leq 2N_{l-1}$ by definition of $N_{l-1}$, we conclude that

\begin{equation*}
\begin{split}
d(\beta(s), \beta(t)) & \geq \frac{1}{\sqrt{2}} d_{l^1}(\beta(s), \beta(t))\\
& \geq \frac{1}{\sqrt{2}} \left( \frac{N_{l-1}}{2} + s - T - \frac{1}{2}  \right)\\
& \geq \frac{1}{8\sqrt{2}} \left| T - t + s - T - \frac{1}{2} \right|\\
& = \frac{1}{8\sqrt{2}} |s - t| - \frac{1}{16\sqrt{2}}.
\end{split}
\end{equation*}

We conclude that $\beta$ satisfies the $(8\sqrt{2},1)$-quasi-geodesic inequalities for any $0 \leq t \leq s$, which concludes the proof of the claim
\end{proof}

We now return to our open neighbourhood $U_{o, k}$ of $\xi$. Since every hyperplane in $Y_0$ is contained in one block of $Y_0$, it can either intersect flats in $\mathcal{A} \cup \mathcal{B}$ or in $\mathcal{B} \cup \mathcal{C}$. Without loss of generality, $k$ only intersects flats in $\mathcal{A} \cup \mathcal{B}$. Since $\gamma$ crosses infinitely many flats of all three types, there has to be an $L$, such that $k$ separates $o$ from $F_L$. In particular, the concatenation $\beta_1 * \dots * \beta_L$ crosses $k$. We can extend this concatenation to a contracting $(8\sqrt{2},1)$-quasi-geodesic $\beta_k'$, which provides us with a point $\eta_k := [\beta_k'] \in \partial_M Y_0$. Note that $[\beta_k'] \neq [\beta]$, in fact, $\beta$ is not a contracting quasi-geodesic. However, $\beta_k'$ is a quasi-geodesic with the property that $\beta_k' \cap N_{\Delta}(\gamma \diagdown B_R(o)) = \emptyset$ since $R > \Delta$. If we choose $\Delta > \kappa(\rho_{\xi}, 8\sqrt{2}, 1)$, we conclude that $\eta_k \in U_{o, k} \diagdown U_{o, R}(\xi)$. In summary: Given $R \gg \kappa(\rho_{\xi}, 8\sqrt{2},1)$ and any hyperplane $k$, we have found a point $\eta_k = [\beta_k'] \in U_{o, k} \diagdown U_{o, R}(\xi)$. This implies that $\Hyp \neq \FQ$, which completes the proof of the proposition.
\end{proof}

\begin{rem}
Note that we can chose $\Delta$ rather freely in our construction. In particular, it is not possible to adapt the number $\kappa(\rho, K, C)$ in Cashen-Mackays definition of fellow-traveling to make the constructed paths $\beta_k'$ fellow-travel along $\gamma$.
\end{rem}

\begin{rem}
The construction in the proof of Proposition \ref{prop:quasi-geodesic topology} can be done for most points in the Morse boundary of $Y_0$, although the construction becomes more messy, as the geodesic $\gamma$ becomes more complicated. We see that $\FQ$ and $\FG$ provide different open neighbourhoods around nearly every point of $\partial_M Y_0$.
\end{rem}

%---------------------------------------------------------------------------------------------------------------------------------------

%METRIZING THE MORSE BOUNDARY OF RAAGS

%---------------------------------------------------------------------------------------------------------------------------------------

\section{Defining a metric on the Morse boundary of Right-angled Artin groups} \label{sec:Metrizing}

Let $\Gamma$ be a finite graph with no multiple edges and no self-loops for the remainder of this section. Denote by $A_{\Gamma}$ the induced RAAG and by $Y_{\Gamma}$ the universal covering of the Salvetti complex. More generally, let $Y$ be a locally finite $\CAT$ cube complex such that for every Morse geodesic ray $\gamma$, there exists $rÊ\geq 0$ and an infinite family of strongly separated hyperplanes $h_i$ crossed by $\gamma$, such that for $p_i := \gamma \cap h_i$ we have $d(p_i, p_{i+1}) \leq r$. We say that \emph{all Morse geodesic rays cross an infinite sequence of strongly separated hyperplanes}.

We will show that, whenever $Y$ satisfies this property, the topology $\Hyp$ on its Morse boundary is metrizable and has a nice description. The metric will depend on the choice of a base point $o$, so we will obtain a family of metrics $d_o$. In fact, these metrics even induce different cross ratios on the Morse boundary, as we will see further below.

Before we define the actual metric, recall the following fact, which is an essential part of the proof of the Urysohn metrization theorem (cf. \cite{Royden}).

Let $(Y, \mathcal{T})$ be a topological space, $\{ U_n \}$ a countable basis of $\mathcal{T}$ and $f_n : Y \rightarrow [0, e^{-n}]$ continuous maps such that $\supp(f_n) \subset U_n$ and for every $y \in Y$, $U \in \mathcal{T}$ with $y \in U$ there exists some $f_n$ with $f_n(y) \neq 0$ and $\supp(f_n) \subset U$. Then the map

\[ d(x,y) := \sup_n( | f_n(x) - f_n(y) |) \]

is a metric and its induced topology is equal to $\mathcal{T}$. In fact, we can even have finitely many sets $U_{n,1}, \dots, U_{n, l_n}$ and finitely many functions $f_{n, 1}, \dots f_{n, l_n} : Y \rightarrow [0, e^{-n}]$ with the properties above and the construction still yields a metric that induces $\mathcal{T}$. We will explicitly choose the sets $U_n$ and construct functions $f_n$. We start by proving certain properties that will justify our choice.

\begin{lem} \label{lem:AuxiliaryforMetrization}
If $Y$ is a locally finite $\CAT$ cube complex and all Morse geodesic rays in $Y$ cross an infinite sequence of strongly separated hyperplanes, then the family $\{ U_{o, k} \}_{k \in \mathcal{W}(Y)}$ is a countable basis of $\Hyp$. Furthermore, $U_{o, k}$ is open and closed in $\Hyp$.
\end{lem}

\begin{proof}
The countability of $\{ U_{o, k} \}_k$ follows from the local finiteness of $Y$. Next, we will show that $\{ U_{o, k} \}_k$ is a basis. Let $h_1, \dots, h_n$ be a collection of hyperplanes such that $U_{o, h_1, \dots, h_n} \neq \emptyset$. Let $\xi \in U_{o, h_1, \dots, h_n}$ and $\gamma \in \xi$ a geodesic representative based at $o$. We find a sequence of strongly separated hyperplanes $(k_i)_i$ that are crossed by $\gamma$. Consider the first hyperplane $k_i$ that is crossed by $\gamma$ after it has crossed all $h_j$. Since all $k_i$ are strongly separated, $k_{i+1}$ cannot cross $h_j$ for any $j$ and thus, $\xi \in U_{o, k_{i+1}} \subset U_{o, h_1, \dots, h_n}$. We conclude that $\{ U_{o, k} \}_k$ is a basis.

Finally, we show that $U_{o, k}$ is closed in $\Hyp$. Let $\xi \in \partial_M Y \diagdown U_{o, k}$ and $\gamma \in \xi$ the geodesic representative based at $o$. There exists a sequence of strongly separated hyperplanes $(k_i)_i$ that is crossed by $\gamma$. Since $\gamma$ does not cross $k$, there can be at most one $k_i$ that crosses $k$. Let $k_I$ be the first hyperplane in $(k_i)_i$ after the one that crosses $k$ (pick any $k_I$ if no $k_i$ crosses $k$). We have that $\xi \in U_{o, k_I} \subset \partial_M Y \diagdown U_{o, k}$. We conclude that $U_{o, k}$ is closed.
\end{proof}

Lemma \ref{lem:AuxiliaryforMetrization} allows us to make the construction from Urysohns metrization theorem very explicit. We can choose $\{ U_{o, k} \}_k$ as our countable basis and define 

\begin{equation*}
f_k(\xi) := \begin{cases} e^{- \# \mathcal{W}(o, k)} & \text{ if $\xi \in U_{o, k}$},\\
0 & \text{ else}
\end{cases}
\end{equation*}

Recall that $\mathcal{W}(o, k)$ is the set of hyperplanes that separates $o$ from $k$. Then
\[ d_o(\xi, \eta) := \sup_k( | f_k(\xi) - f_k(\eta) | ) \]

is a metric and its induced topology is equal to the topology generated by $\{ U_{o, k} \}_k$, i.\,e. $\Hyp$.

There is a more convenient way to describe $d_o$. Let $\xi$, $\eta \in \partial_M Y$, $o \in Y^{(0)}$. Since $f_k(\xi) \in \{ 0, e^{- \# \mathcal{W}(o, k)} \}$, we see that $f_k(\xi) - f_k(\eta) = 0$ unless one of the two lies in $U_{o, k}$ and the other one does not. Therefore, if we define
\[ \mathcal{W}(\xi, \eta) := \{ h \in \mathcal{W}(Y) \vert \text{ Exactly one geodesic representative of $\xi$ and $\eta$ based at $o$ crosses $h$} \}, \]

we have
\[ d_o(\xi, \eta) = e^{- \inf \{ \# \mathcal{W}(o,k) | k \in \mathcal{W}(\xi, \eta) \}}. \]

To abbreviate, we define
\[ [\xi | \eta ]_o := \inf( \#\mathcal{W}(o,k) | k \in \mathcal{W}(\xi, \eta)) \]

and obtain
\[ d_o(\xi, \eta) = e^{-[\xi | \eta]_o}. \]

We use the expression $[\xi | \eta]_o$ as an analogue of the Gromov product in hyperbolic spaces. Note that $[\xi | \eta]_o$ is very different from the product $(\xi | \eta)_o := \# \mathcal{W}(o, \{ \xi, \eta \})$ which was used in \cite{BeyFioMedici}.

We define the cross ratio of four points $w, x, y, z$ with respect to $[\cdot | \cdot]_o$ by
\[ cr_o(w,x,y,z) := [w | x]_o + [y | z]_o - [wÊ| y]_o - [x | z]_o. \]

In addition, we define
\[ [w,x,y,z] := (w | x)_o + (y | z)_o - (wÊ| y)_o - (x | z)_o, \]

which is the base point-independent cross ratio introduced in \cite{BeyFioMedici}. Cross ratios often appear on boundaries in a very natural way (cf. \cite{Paulin} and \cite{BeyFioMedici}). One particular desirable feature of cross ratios on boundaries is that they should be independent of the choice of any base point (in contrast to the construction of metrics). The example below shows us that:

\begin{enumerate}

\item[1)] The difference $\vert cr_o(\cdot, \cdot, \cdot, \cdot) - [\cdot, \cdot, \cdot, \cdot] \vert$ is unbounded. 

\item[2)] The cross ratio $cr_o$ is not independent of the choice of $o$.

\end{enumerate}

\begin{exam}

Consider the `tree of $3$-dimensional flats' that corresponds to the RAAG
\[ \mathbb{Z}^3 * \mathbb{Z} = \langle a, b, c, d | [a,b] = [b,c] = [a,c] = 1 \rangle . \]

Denote the Salvetti complex that belongs to the graph $\Gamma = K_3 \cup \{ d \}$ by $Y$. The Cayley-graph of the representation given above can be embedded into $Y$ by an embedding that respects the cube complex structure. Let $o$ be the image of $1 \in \mathbb{Z}^3 * \mathbb{Z}$ under a chosen embedding. The choice of $o$ allows us to represent elements in the visual boundary by infinite words in $a$, $b$, $c$, $d$ and their inverses. Put

\[ w_n := a^n d^{\infty}, \]
\[ x_n := a^n b d^{\infty}, \]
\[ y := a^{-1} b^{-1} d^{\infty}, \]
\[ z := a^{-1} b^{-1} c d^{\infty}. \]

Clearly,

\[ [w_n | x_n]_1 = [w_n | y]_1 = [w_n | z ]_1 = [x_n | z]_1 = [y | z]_1 = 0, \]

while

\[ (w_n | x_n)_1 = n \]
\[ (w_n | y)_1 = (w_n | z )_1 = (x_n | z)_1 = (y | z)_1 = 0. \]

We see that

\[ cr_1(w_n, x_n, y, z) = 0 \]
\[ [w_n, x_n, y, z] = n. \]

This proves 1). We see that the metric $d_o$ provides us with a cross ratio that is very different from the cross ratio introduced in \cite{BeyFioMedici}. Furthermore, changing the base point also changes the cross ratio $cr_o$. For example,
\[ cr_{c^{-m}}(w_n, x_n, y, z) = 0 + m - 0 - 0 = m \neq cr_1(w_n, x_n, y, z). \]

This proves 2).

\end{exam}

%---------------------------------------------------------------------------------------------------------------------------------------

%HYP AND THE VISUAL TOPOLOGY

%---------------------------------------------------------------------------------------------------------------------------------------

\section{$\Hyp$ and the visual topology} \label{sec:VisualTopology}

In this section, we will connect the topologies $\Hyp$ and $\FG$ with the visual topology on the Morse boundary of $\CAT$ cube complexes and the quotient topology coming from the Roller boundary. All together, this will provide us with a new way to tackle the question whether the subspace topology of the visual topology on the Morse boundary is a quasi-isometry-invariant. In particular, our understanding of various cubulable groups provides us with many examples to study this question with. We start by noting a result by Cashen-Mackay and one by Beyrer-Fioravanti.

\begin{lem}[Proposition 7.3 from \cite{CashenMackay}] \label{lem:CashenMackay}
Let $Y$ be a $\CAT$ space. Then $\FG$ agrees with the subspace topology induced by the visual topology on $\partial_M Y$.
\end{lem}

If $Y$ is a locally finite $\CAT$ cube complex, the Roller boundary also induces a topology on $\partial_M Y$. Specifically, consider the projection map $\Phi : \partial_{R, M} Y \rightarrow \partial_{M} Y$ introduced in section \ref{subsec:CCCs}. The Roller boundary induces a subspace topology on $\partial_{R, M} Y$. The following result is part of Theorem 3.10 in \cite{BeyFio}.

\begin{thm}[\cite{BeyFio}] \label{thm:RollerTopology}
Let $Y$ be a uniformly locally finite $\CAT$ cube complex and equip $\partial_M Y$ with the visual topology. Then the map $\Phi : \partial_{R, M} Y \rightarrow \partial_M Y$ is surjective and continuous.

In particular, the map $\phi : \partial_{R, M} Y \diagup \sim \; \rightarrow \partial_M Y$ where asymptotic points in $\partial_{R, M} Y$ are identified and the quotient is equipped with the quotient topology, is a homeomorphism.
\end{thm}

The topology on the Roller boundary used in this Theorem is connected to several rigidity results (cf. \cite{BeyFioMedici}, \cite{BeyFio2}). However, these rigidity results use the cross ratio $[\cdot, \cdot, \cdot, \cdot]$, which is different from $cr_o$, as we have seen in section \ref{sec:Metrizing}.

Combining Theorem \ref{thm:TheoremA}, Lemma \ref{lem:CashenMackay} and Theorem \ref{thm:RollerTopology}, we obtain Corollary \ref{cor:Corollary1}. This gives us a new approach to tackle the question, whether the Morse boundary with the visual topology is a quasi-isometry invariant for uniformly locally finite $\CAT$ cube complexes that admit a cocompact action by isometries. For spaces, where $\FG = \FQ$, this follows from the quasi-isometry invariance of $\FQ$. Theorem \ref{thm:TheoremB} provides an example, where this is not the case (and suggests the existence of many others). We finish by presenting the naive attempt to prove quasi-isometry invariance of the topology $\Hyp$ and by illustrating an obstruction to this invariance.

Let $Y$, $Y'$ be two $\CAT$ cube complexes with dimension uniformly bounded by $n$, $X$ and $X'$ their respective Morse boundaries, $F : Y \rightarrow Y'$ a $(K,C)$-quasi-isometry between them and $f : X \rightarrow X'$ the induced bijection of the Morse boundaries. Pick a base point $o \in Y$ and let $o'$ be a vertex closest to $F(o)$. Let $k'$ be a hyperplane in $Y'$ inducing a open set $U_{o', k'} \subset X'$ and suppose, $f(\xi) = \zeta \in U_{o',k'}$. To prove continuity of $f$, we need to find hyperplanes $h_1, \dots, h_l$ such that $\xi \in U_{o, h_1, \dots, h_l}$ and $f(U_{o, h_1, \dots, h_l}) \subset U_{o', k'}$.

Consider the unique geodesic representative $\gamma$ of $\xi$ based at $o$. Concatenating the geodesic segment from $o'$ to $F(o)$ with $F \circ \gamma$ provides us with a $(K, C + \frac{1}{2} \sqrt{n})$-quasi-geodesic representative of $f(\xi)$ which -- by assumption -- crosses the hyperplane $k'$. We need to show that for $\eta \in U_{o, h_1, \dots, h_l}$, the geodesic representative of $f(\eta)$ based at $o'$ crosses $k'$. For this, consider the geodesic representative $\delta$ of $\eta$ in $Y$ based at $o$. Since $\gamma$ is Morse, we can choose $h_1, \dots, h_l$ such that for all $\eta \in U_{o, h_1, \dots, h_l}$ its geodesic representative $\delta$ stays close to $\gamma$ for a long distance. While this implies that the image $F \circ \delta$ crosses the hyperplane $k'$ (if $\delta$ and $\gamma$ fellow-travel for sufficiently long), it is not obvious at all that $F \circ \delta$ does not travel back and crosses $k'$ again, implying that the geodesic representative of $f(\eta)$ does not cross $k'$ at all (cf. Figure \ref{fig:Back-Travel}). Note that, if $F \circ \delta$ does travel back, it cannot stay close to $F \circ \gamma$ while doing so, as it is a quasi-geodesic whose constants are controlled by $F$.

Proving that $f$ is continuous proves that $h_1, \dots, h_l$ can be chosen such that this kind of back-traveling does not occur. A quasi-isometry that exhibits such back-traveling would provide a counter-example to quasi-isometry invariance. Thus, we finish with the following

\begin{figure}%Picture of the quasi-geodesic
\begin{tikzpicture}[scale=1.50]
%Left side
\draw [fill] (-5.5,0) -- (-2.9,0);
\draw (-5.5,0) to [out = 2, in = 181] (-3.9, 0.15);
\draw (-3.9, 0.15) to [out = 2, in = 270] (-3.65, 1.3);
\draw (-5.5,0) to [out = 1, in = 182] (-3.6, 0.1);
\draw (-3.6, 0.1) to [out = 1, in = 270] (-3.35, 1.3);
\draw (-5.5,0) to [out = 0.5, in = 180.5] (-3.3, 0.05);
\draw (-3.3, 0.05) to [out = 0.5, in = 270] (-3.05, 1.3);
\draw [fill] (-4, 1.2) -- (-4, -0.3);
\node [below] at (-4,-0.3) {$h_1$};
\draw [fill] (-4.2,1) -- (-3.75, -0.4);
\node [below] at (-3.75, -0.4) {$h_l$};
\node [right] at (-2.9, 0) {$\xi$};
\node [above] at (-3.65, 1.3) {$\eta_1$};
\node [above] at (-3.35, 1.3) {$\eta_2$};
\node [above] at (-3.05, 1.3) {$\eta_3$};
\node [left] at (-5.5,0) {$o$};
%Arrow
\draw [->] (-2.5,0.5) -- (-1.1, 0.5);
\node [below] at (-1.8, 0.5) {$F$};
%Right side
%First quasi-geodesic
\draw (-1, 0) to [out = 0, in = 200] (0, 0);
\draw (0, 0) to [out = 20, in = 160] (0.5, 0);
\draw (0.5, 0) to [out = -20, in = 170] (1, 0);
\draw (1, 0) to [out = -10, in = 180] (2, 0);
\node [right] at (2, 0) {$f(\xi)$};
%Second quasi-geodesic
\draw (-1, 0) to [out = 20, in = 210] (0, 0.2);
\draw (0, 0.2) to [out = 30, in = 270] (1, 0.4);
\draw (1, 0.4) to [out = 90, in = -10] (0.5, 0.7);
\draw (0.5, 0.7) to [out = 170, in = 20] (0.3, 0.7);
\draw (0.3, 0.7) to [out = 200, in = -50] (0, 0.9);
\draw (0, 0.9) to [out = 130, in = 10] (-0.2, 1.1);
\draw (-0.2, 1.1) to [out = 190, in = 270] (-0.5, 1.3);
\node [above] at (-0.6, 1.3) {$f(\eta_1)$};
%Third quasi-geodesic
\draw (-1, 0) to [out = 15, in = 210] (0, 0.1);
\draw (0, 0.1) to [out = 30, in = 270] (1.3, 0.3);
\draw (1.3, 0.3) to [out = 90, in = -10] (0.6, 0.9);
\draw (0.6, 0.9) to [out = 170, in = 20] (0.3, 1);
\draw (0.3, 1) to [out = 200, in = -50] (0.1, 1.1);
\draw (0.1, 1.1) to [out = 130, in = 10] (0, 1.2);
\draw (0, 1.2) to [out = 190, in = 270] (-0.2, 1.6);

\node [above] at (-0.3, 1.6) {$f(\eta_2)$};
%Fourth quasi-geodesic
\draw (-1, 0) to [out = 10, in = 210] (0, 0.05);
\draw (0, 0.05) to [out = 30, in = 180] (1, 0.05);
\draw (1, 0.05) to [out = 0, in = 270] (1.5, 0.1);
\draw (1.5, 0.1) to [out = 90, in = -10] (0.8, 1.1);
\draw (0.8, 1.1) to [out = 170, in = -20] (0.5, 1.3);
\draw (0.5, 1.3) to [out = 160, in = -50] (0.4, 1.4);
\draw (0.4, 1.4) to [out = 130, in = 270] (0.3, 1.6);

\node [above] at (0.3, 1.6) {$f(\eta_3)$};
%Hyperplane
\draw [fill] (0.5, 1.5) -- (0.5, -0.3);
\node [below] at (0.5, -0.3) {$k'$};
\end{tikzpicture}
\caption{Is there a quasi-isometry $F$ such that the images of fellow-traveling, contracting geodesics look like this?}
\label{fig:Back-Travel} 
\end{figure}
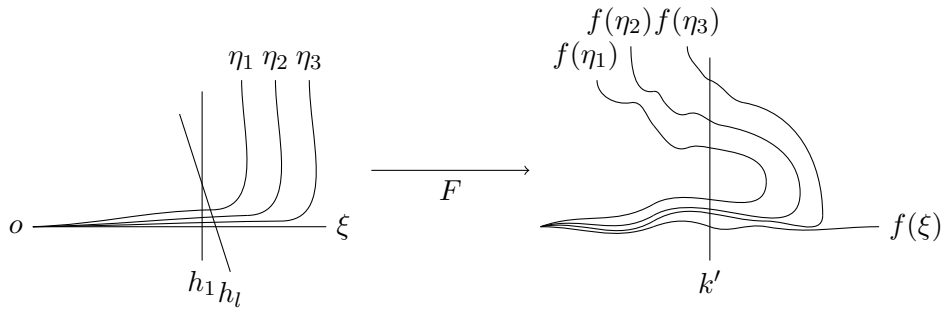

\begin{quest*}
Is there a quasi-isometry between `nice' $\CAT$ cube complexes displaying the `back-traveling' described above?
\end{quest*}

\bibliographystyle{alpha}
\bibliography{mybib}

\end{document}